\documentclass[preprint,12pt,authoryear]{article}
\usepackage[utf8]{inputenc}

\usepackage[top=2.5cm,bottom=2.5cm,left=3cm,right=2.5cm]{geometry}

\usepackage[linkcolor = red, colorlinks = true]{hyperref}

\usepackage[dvips]{epsfig}
\usepackage{color}
\usepackage{tikz-cd} 
\usepackage{amsgen, amstext,amsbsy,amsopn,amssymb, amsthm,mathrsfs,mathptmx,amsfonts,amssymb,amscd,amsmath,euscript,enumitem,verbatim,calc,enumitem,mathtools}

\usepackage{url}
\usepackage{hyperref}

\usepackage[normalem]{ulem}

\usepackage[multiple]{footmisc}
\usepackage{bigfoot}

\newcommand\blfootnote[1]{%
  \begingroup
  \renewcommand\thefootnote{}\footnote{#1}%
  \addtocounter{footnote}{-1}%
  \endgroup
}
\usepackage{soul,xcolor}

\theoremstyle{plain}
\newtheorem{theorem}{Theorem}[section]

\newtheorem{corollary}[theorem]{Corollary}
\newtheorem{lemma}[theorem]{Lemma}
\newtheorem{notation}[theorem]{Notation}

\theoremstyle{definition}
\newtheorem{definition}[theorem]{Definition}

\newtheorem{example}[theorem]{Example}
\theoremstyle{remark}
\newtheorem{remark}[theorem]{\bf Remark}
\begin{document}	     	
\author{
  Ambily A. A.\\
  \text{Department of Mathematics},\\ 
  \textit{Cochin University of Science and Technology},\\
  \textit{Kochi, Kerala, India, 682022}\\
  \texttt{ambily@cusat.ac.in}\\
  \and
  Gayathry Pradeep\footnote{Corresponding author}\\
  \text{Department of Mathematics},\\
  \textit{Cochin University of Science and Technology},\\
  \textit{Kochi, Kerala, India, 682022}\\
  \texttt{gayathrypradeep221b@gmail.com}\\}
  \title{Solvability of the ${\rm SK_1}$-analog of the orthogonal groups}

\maketitle     	
\begin{abstract}
We prove the dilation principle for the relative Dickson-Siegel-Eichler-Roy (DSER) elementary orthogonal group and using the dilation principle we prove the Quillen's analog of the local-global principle for the group. Applying the relative local-global principle, we prove the solvability and nilpotency of the ${\rm SK_1}$-analog of the orthogonal groups and study the homotopy and commutativity principle for odd elementary orthogonal groups.
\end{abstract}

\blfootnote{{\it Keywords:} Quadratic modules, Orthogonal groups, DSER elementary orthogonal group, Local-global principle, Solvable groups}
 
\blfootnote{{\it 2020 Mathematics Subject Classification:}19G99, 13C10, 11E70, 20H25, 13H99, 20F16}

    \date{\today}
	
	\setstcolor{red}

    \section{Introduction}
	The work of A. A. Suslin in \cite{Suslin1977} established the significant result on the normality of the elementary linear group ${\rm E}_n(R)$ in the general linear group ${\rm GL}_n(R)$, for $n\geq 3$. This led to the definition of the quotient group ${\rm K}_{1,n}(R)$ as $\frac{{\rm GL}_n(R)}{{\rm E}_n(R)}$. As the derived subgroup gives an insight into the commutativity of the given group, the study of the derived series and the solvability of the given group are of equal importance as that of the abelian groups. In \cite{Bak1991,vanderKallen1989}, A. Bak and W. van der Kallen investigated the solvability and nilpotency of ${\rm K}_{1,n}(R)$ and A. Bak used the localization-completion method to prove the following remarkable theorem in \cite{Bak1991}. 

\begin{theorem}[\cite{Bak1991}]
Let $R$ be a commutative ring with finite Bass-Serre dimension $\delta(R)$, and $S$ be an associative quasi-finite $R$-algebra. Then the group ${\rm K}_{1,n}(S)$ with $n\geq 3$ is nilpotent by abelian and the nilpotency class is at most $\delta(R)+2-n$.
\end{theorem}
\vspace{2mm}

Subsequent to the simultaneous and celebrated proofs of the {\it Serre's problem on projective modules} by A. A. Suslin and D. Quillen in \cite{Suslin1977, Quillen1976}, a technique called the local-global principle (LGP) has emerged to generalize the results of elementary groups over local rings to elementary groups on projective modules over arbitrary rings. Later, this study was carried forward by Suslin-Kopeiko in \cite{SuslinKopeiko1977} by validating the results to orthogonal groups for proving stabilization of orthogonal groups and cancellation theorems for quadratic modules over polynomial rings. Quillen's analogue of the local-global principle for symplectic groups was proved in \cite{Kopeiko1978} by V. I. Kopeiko shortly. In \cite{Abe1983}, E. Abe proved a version of local-global principle for general Chevalley groups.  In \cite{Bass1968}, H. Bass introduced the transvection groups and its elementary subgroups. Later, the local-global principle for elementary subgroups of the transvection groups were proved by A. Bak, R. Basu, and Ravi A. Rao in \cite{BakBasuRao2010}. Local-global principle for relative elementary transvection groups was proved by R. Basu, Ravi A. Rao, and R. Khanna in \cite{BasuRaoKhanna2020}. Generalization of the local global-principle for the odd unitary group and the isotropic reductive groups are proved by A. Petrov and A. K. Stavrova in \cite{Petrov2005} and \cite{PetrovStavrova}. The normality of the elementary subgroups and the local-global principle of the classical groups and classical like groups facilitated in generalizing the solvability and nilpotency results to these groups as well. In \cite{Roozbeh2002, Roozbeh2003, BAKRHNV2009}, R. Hazrat, N. Vavilov and A. Bak generalized the above result to quadratic modules and the Chevalley groups using Bak's localization-completion method. In general, the nilpotent structure for Chevalley groups studied in \cite{Stepanov2016} by A. Stepanov using a localization technique called {\it universal localization}.

\vspace{2mm}

In \cite{Roy1968}, Amit Roy extended the definition of Eichler transformations on quadratic spaces over fields to the orthogonal sum of a non-degenerate quadratic space and a hyperbolic space over a commutative ring. The group generated by these transformations is called the Dickson-Siegel-Eichler-Roy (DSER) elementary orthogonal group. In \cite{AmbilyRao2020}, A. A. Ambily and R. A. Rao established Quillen's local-global principle for the DSER elementary orthogonal group. The DSER elementary orthogonal group coincides with the elementary orthogonal group ${\rm EO}_n(R)$ on a free quadratic $R$-space with standard hyperbolic form. Thus, the group is a generalization of the elementary orthogonal group. Similar to the classical groups, we can consider the relative DSER elementary orthogonal group with respect to an ideal $I$ of $R$. This article extends the local-global principle to the relative DSER elementary orthogonal group. 

\vspace{2mm}

In \cite{RRaoSS2017}, R. A. Rao and S. Sharma studied the solvability and nilpotency of the reduced ${\rm K}_1$ of the linear, symplectic and even orthogonal groups. For the even orthogonal groups they have obtained the following result using local-global principle:

\begin{theorem}[\cite{RRaoSS2017}]\label{RRSSSol}
    Let R be a local ring. Then the quotient group $\frac{{\rm SO}_{2m}(R[X])}
{{\rm EO}_{2m}(R[X])}$ is solvable of length at most $2$.
\end{theorem}

\vspace{2mm}

 In \cite{RRaoSS2017}, the authors proved that the group $\frac{{\rm SL}_{n}(R[X])}
{{\rm E}_{n}(R[X])}$ is abelian for a local ring $R$. Now, we need to check whether the group $\frac{{\rm SO}_{2m}(R[X])}
{{\rm EO}_{2m}(R[X])}$ is abelian or not. Nevertheless, we were able to prove the odd orthogonal analog of Theorem~\ref{RRSSSol} by using the relative local-global principle of the DSER elementary orthogonal group. Furthermore, we establish the solvability and nilpotency of ${\rm SK}_1$-analog of the relative DSER elementary orthogonal group as follows:
\begin{theorem}\label{DSERSolvability}
		Let $(Q,q)$ be a quadratic $R$-space of rank $n \geq 1 $ and $P$ be a finitely generated projective $R$-module of rank $m\geq2$. Let $M=Q\perp \mathbb{H}(P)$. If for every maximal ideal $\mathfrak{m}$ of $R,\; q_{\mathfrak{m}}\perp h_{\mathfrak{m}}=\tilde{\phi}_{n+2m}$, where $h$ is the hyperbolic form on $\mathbb{H}(P)$, then the group $\frac{{\rm SO}_{(R[X],(X))}(M[X],XM[X])}{{\rm EO}_{(R[X],(X))}(Q[X],\mathbb{H}(P[X]))}$ is solvable of length at most $2$.
	\end{theorem}

\begin{theorem}\label{DSERNilpotency}
	Let $Q$ be a quadratic $R$- space of rank $n\geq 1$ and $P$ be a finitely generated projective module of rank $m\geq 2$. Let $M=Q\perp \mathbb{H}(P)$. If for all maximal ideal $\mathfrak{m}$ of $R,\; q_{\mathfrak{m}}\perp h_{\mathfrak{m}}=\tilde{\phi}_{n+2m}$, where $h$ is the hyperbolic form on $\mathbb{H}(P)$, then the nilpotency class 
 of the group $\frac{{\rm SO}_{(R[X],(X))}(M[X],XM[X])}{{\rm EO}_{(R[X],(X))}(Q[X],\mathbb{H}(P[X])}$ at most 2.
\end{theorem}
In \cite{RRaoSS2017}, the authors studied the commutativity principle for the special linear group and the symplectic group. For the special linear group, the result says that the product of matrices is commutative up to an elementary matrix if the multiplicand is homotopic to identity. This result is valid for the symplectic group as well. The even orthogonal analog of the above result was proved in \cite{RRaoSS2022} in a different approach using the relative local-global principle. The authors proved the following theorem:

\begin{theorem}[{\cite[Theorem~4.13]{RRaoSS2022}}]\label{Homotopy}
	Let $m\geq 3$ and let $\alpha \in {\rm SO}_{2m}(R)$ be homotopic to identity. Then the commutator $[\alpha,\beta  ]\perp I_2 \in {\rm EO}_{2m+2}(R)$, for all $\beta \in {\rm O}_{2m}(R)$.
	
\end{theorem}

We extend Theorem \ref{Homotopy} to the odd orthogonal group as follows.

\begin{theorem}\label{homotopytheorem}
	
 Let $m\geq 3$ and let $\alpha \in {\rm SO}_{2m-1}(R)$ be homotopic to identity. Then the commutator $[\alpha,\beta  ]\perp I_2 \in {\rm EO}_{2m-1}(R)$, for all $\beta \in {\rm O}_{2m-1}(R)$.
	
\end{theorem}

	\section{Preliminaries}

	Throughout this paper, $R$ denotes a  commutative Noetherian ring with unity in which $2$ is invertible. A quadratic form on an $R$-module $M$ is a function $q:M\to R$ which  satisfies the following conditions. (i) $q(rx)=r^2x \; {\rm for } \; r\in {R}$ and (ii) $q(x+y)-q(x)-q(y)$ is a bilinear form on $M\times M$ to $R$. The associated bilinear form of $q$ given by the above expression (ii) is denoted by $\langle \ ,\ \rangle_q$. 

\vspace{1mm}

A bilinear form $\langle \ ,\ \rangle$ on a quadratic module $M$ is said to be non-degenerate if $ \langle x,y \rangle =0$ for all $y \in M$ implies $ x=0$. The quadratic form is said to be non-degenerate if the associated bilinear form is non-degenerate.

\vspace{1mm}

 A quadratic $R$-space $M$ is a finitely generated projective $R$-module with a non-degenerate quadratic form $q$ on it. A transformation $\sigma$ on $M$  such that $\langle \sigma x, \sigma y \rangle_q=\langle x, y \rangle_q $ is called an {\it orthogonal transformation}. The group generated by the orthogonal transformations on $M$ is {\it the orthogonal group} and is denoted by ${\rm O}_{R}(M)$. Let $Id$ denote the identity element transformation group ${\rm O}_{R}(M)$.
 \vspace{1mm}
 Let $I$ be an ideal of the commutative ring $R$. The kernel of the natural homomorphism $\pi:{\rm O}_R(M) \rightarrow {\rm O}_R(M/IM)$ is called the relative orthogonal group on $M$ with respect to an ideal $I$ of $R$ and is denoted by ${\rm O}_{(R,I)}(M,IM)$. The special orthogonal group ${\rm SO}_{(R,I)}(M,IM)$ is the kernel of the determinant homomorphism ${\it det }:{\rm O}_{(R,I)}(M,IM) \rightarrow R^*$.
		\vspace{1mm}
	We will now recall the definitions and results from \cite{HahnOMeara1989, Klingenberg1961,  Roy1968, SuslinKopeiko1977}.
 
	\begin{definition}
		Let $P$ be a finitely generated projective $R$-module. Then the {\it hyperbolic space} $\mathbb{H}(P)$ is the quadratic space $P\oplus P^*$ with the quadratic form $q((x,f))=f(x)$. Then the associated bilinear form can be obtained as $\langle(x_1,f_1),(x_2,f_2)\rangle =f_2(x_1)+f_1(x_2)$ for $x_1,x_2\in P$ and $f_1,f_2 \in P^*$, where $P^*$ is the dual space ${\rm Hom}_R(P,R)$ of $P$.
	\end{definition}
 
 \vspace{1mm}
 
The classical orthogonal group is defined to be $${\rm O}_{n}(R):=\{\alpha \in {\rm GL}_n(R)|\alpha^t\phi\alpha=\phi\},$$ for some non-degenerate $n\times n$ symmetric matrix $\phi$. If we consider $Q$ to be a free quadratic $R$-space of rank $n\geq 1$ and the projective module $P$ to be $R^2$, then any orthogonal transformation can be viewed as an orthogonal matrix $\alpha$ with $\alpha^t\phi \alpha=\phi$ for the symmetric matrix $\phi $ corresponding to the symmetric bilinear form on $Q\perp\mathbb{H}(R)^2$. That is, when $M=Q\perp\mathbb{H}(R)^2$, the orthogonal group ${\rm O}_R(M)$ coincides with the classical orthogonal group ${\rm O}_{n+4}(R)$. The special orthogonal group ${\rm SO}_n(R)$ is the subgroup of ${\rm O}_n(R)$ consisting of matrices having determinant one. 

\vspace{2mm}

 Next, we define the relative orthogonal subgroup of the classical orthogonal group.

\begin{definition}[{{\it Relative orthogonal group}}]
	Let $R$ be a ring and $I$ be an ideal of $R$. The kernel of the natural homomorphism $\pi :{\rm  O}_n(R) \rightarrow {\rm O}_n(R/I) $ is the relative orthogonal group and is denoted by ${\rm O}_n(R,I)$. The relative special orthogonal subgroup ${\rm SO}_n(R,I)$ is the subgroup of ${\rm O}_n(R,I)$ consisting of matrices having determinant one.
\vspace{2mm}

  In this article, we consider the symmetric matrix 
$$\tilde{\phi}_n:=\begin{cases} \sum_{i=1}^{r}e_{2i-1,2i}+\sum_{i=1}^{r}e_{2i,2i-1} \ {\rm for}\  n=2r,\\
\vspace{.25cm}
{2}\perp {\sum_{i=1}^{r}e_{2i-1,2i}+\sum_{i=1}^{r}e_{2i,2i-1}} \ {\rm  for } \ n=2r+1, \end{cases}$$
\noindent where $e_{i,j}$ is the $n\times n$ matrix with $(i,j)^{th}$ entry $1$ and $0$ elsewhere.
\end{definition}

\vspace{2mm}

	In \cite{Roy1968}, Amit Roy generalized certain orthogonal transformations used in the works of L. E. Dickson, C. L. Siegel, M. Eichler, and C. T. C. Wall in \cite{Dickson1958, Siegel1938, Eichler1974, Wall1964}. Roy defined those transformations on the orthogonal sum of a non-degenerate quadratic space and a hyperbolic space over a commutative ring $R$ to prove the cancellation theorem for quadratic $R$-spaces \cite[Theorem~7.2]{Roy1968}. These transformations are known as {\it Dickson-Siegel-Eichler-Roy} ({\it DSER}) {\it elementary orthogonal transformations}.
	
	\begin{definition}[{{\it DSER elementary orthogonal transformations}}]
		Let $(Q,q)$ be a quadratic $R$-space and for some finitely generated projective $R$-module $P$, let $M=Q \perp \mathbb{H}(P)$. Then there is an isomorphism $d_{\langle \ ,\ \rangle_q}:Q \to Q^* $,  defined as $ d_{\langle \ ,\ \rangle_q}(z)(w) = \langle z,w\rangle_q$, for $z,w \in Q$.

  	\vspace{2mm}

		Let $\alpha:Q \to P$ be an $R$-linear map. Let $\alpha^t:P^* \to Q^*$ be the dual map defined by $\alpha^t(\phi_p)=\phi_p\circ\alpha$ for $\phi_p \in P^*$. Now let $\alpha^*:P^* \to Q $ be the map  $\alpha^*=d_{\langle\ ,\ \rangle_q}^{-1}\circ \alpha^t$. Then the orthogonal transformation $E_{\alpha} $ on $M$ is defined by:	$$E_{\alpha}(z,x,f) = (z-\alpha^*(f) ,x+\alpha(z)-\frac{1}{2}\alpha\alpha^*(f), f) \mbox{ for } z \in Q, x \in P \mbox{ and } f \in P^*.$$

  \vspace{2mm}
  
Similarly, for an $R$-linear map $\beta:Q \to P^*$, define the dual map $\beta^t : P^{**} \to Q^*$ by $\beta^t(\phi_p^*)=\phi_p^* \circ \beta$, for $\phi_p^*\in P^{**}$. Now define $\beta^*:P \to Q $ as $\beta^*=d_{\langle \ ,\ \rangle_q}^{-1}\circ \beta^t \circ i$, where  $i$ is the natural isomorphism from $P \to P^{**}$. Then the orthogonal transformation $E_{\beta}^*$ on $M$ is defined as:
$$E^*_{\beta}(z,x,f) = (z-\beta^*(x) ,x, \beta(z)-\frac{1}{2}\beta\beta^*(x)+f), \mbox{ for } z \in Q,x \in P \mbox{ and } f\in P^*. $$
These transformations are called DSER elementary orthogonal transformations. The group generated by these orthogonal transformations is called the {\it DSER elementary orthogonal group} and is denoted by ${\rm EO}_{R}(Q,\mathbb{H}(P))$.
\end{definition}

\begin{definition}[{{\it Elementary orthogonal group}}]

Let $\sigma$ be a permutation on $\{1,2,\dots ,2n\}$ such that $\sigma(2i)=2i-1$ and $\sigma(2i-1)=2i$. For $1\leq i,j \leq n, i\neq j$, and  $z\in R$, consider the orthogonal matrices of the form
		$$oe_{ij}(z)=I_{2n}+e_{ij}(z)-e_{\sigma(j)\sigma(i)}(z).$$
  
	  \vspace{1mm}
   
	The subgroup of ${\rm O}_{2n}(R)$ generated by these matrices is called even elementary orthogonal group and is denoted by ${\rm EO}_{2n}(R)$.
 
    \vspace{2mm}
 
	In \cite{FaselCalmes2015}, J. Fasel and B. Calm\`es have given the generators for the odd elementary orthogonal group. For $z\in R$, the matrices $F_{i}^1,F_{i}^2,F_{i,j}^3,F_{i,j}^4$ and $F_{i,j}^5$ are defined as follows:
 \begin{align*}
		F_{i}^1(z) \ &= \ I_{2n+1} + e_{1,2i+1}(z) - e_{2i,1}(2z)- e_{2i,2i+1}(z^2),\\
		F_i^2(z) \  &=\ I_{2n+1}+e_{1,2i}(z)-e_{2i+1,1}(2z) -e_{2i+1,2i}(z^2),\\
		F_{i,j}^3(z) \ &=\  I_{2n+1}+e_{2i,2j}(z)-e_{2j+1,2i+1}(z),\\
		F_{i,j}^4(z) \ &=\  I_{2n+1} + e_{2i,2j+1}(z)-e_{2j,2i+1}(z),\\
		F_{i,j}^5(z) \ &=\  I_{2n+1} +e_{2i+1,2j}(z)-e_{2j+1,2i}(z).
  \end{align*}

\vspace{2mm}

It is easy to verify that $F_{i,j}^3(z)= [F_{i}^2(z),F_{j}^2(1)]$, $F_{i,j}^4= [F_{i}^1(z),F_{j}^1(1)]$ and $F_{i,j}^5= [F_{i}^1(z),F_{j}^2(1)]$. Thus the group generated by matrices of the form $F_{1}^1(z)$ and $F_i^2(z)$, where $1 \leq i \leq n$ and $z\in R$ is called the {\it odd elementary orthogonal group} ${\rm EO}_{2n+1}(R)$. Clearly, ${\rm EO}_{2n+1}(R)$ is a subgroup of ${\rm O}_{2n+1}(R)$. For an ideal $I$ of $R$, the group ${\rm EO}_{2n}(I) $ is the subgroup of ${\rm EO}_{2n}(R)$, generated by $oe_{ij}(z)$ for $z\in I$. Similarly, ${\rm EO}_{2n+1}(I) $ is the subgroup of ${\rm EO}_{2n+1}(R)$, generated by $F^1_i(z)$, $F^2_i(z)$ for $z\in I$. The {\it relative elementary orthogonal group} ${\rm EO}_{n}(R,I)$ is the normal closure of ${\rm EO}_{n}(I)$ in ${\rm EO}_{n}(R)$. 
\end{definition}

\vspace{1mm}

Similar to the definitions of classical relative elementary groups in \cite{HahnOMeara1989}, we define the relative DSER elementary orthogonal group as follows.

	\begin{definition}[{\it Relative DSER elementary transformations}]
	
 Let $I$ be an ideal of $R$. The DSER elementary orthogonal transformations $E_{\alpha}$ and $E_\beta^*$ for $ \alpha : Q \rightarrow IP $ and $\beta: Q \rightarrow IP^*$ are called the relative DSER elementary orthogonal transformations. The group generated by these transformations is denoted by ${\rm EO}_{I}(Q,\mathbb{H}(P))$. The normal closure of ${\rm EO}_{I}(Q ,\mathbb{H}(P))$ in ${\rm EO}_{R}(Q,\mathbb{H}(P))$  is called the relative DSER elementary orthogonal group and is denoted by ${\rm EO}_{(R,I)}(Q ,\mathbb{H}(P))$.
		
\end{definition}

	\vspace{2mm}
	
	The excision ring method can be used to lift an element of relative group to an element of the absolute group and extend the results of absolute groups to corresponding relative subgroups. The excision ring $R\oplus I$, where $R$ is a ring and $I$ is an ideal of $R$ is defined as follows.
	
\begin{definition}[\cite{AGupta2015},{ \it  Excision Ring}]
		Let $R$ be a ring and $I$ be an ideal of $R$. Define addition and multiplication in the excision ring $R\oplus I $ as follows.
		$$(r_1,i_1)+(r_2,i_2):=(r_1+r_2,i_1+i_2),$$
		$$(r_1,i_1)\cdot(r_2,i_2):=(r_1r_2,r_1i_2+r_2i_1+i_1i_2),$$	
 for $r_1,r_2 \in R $ and $i_1,i_2 \in I$. Then $R\oplus I$ forms a ring with the above operations.
 \end{definition}

 Let $f:R\oplus I \rightarrow R$ be the natural projection homomorphism defined by $f(r,i)=r+i$. Then $f$ induces a natural homomorphism $f:{\rm O}_n(R\oplus I) \rightarrow {\rm O}_n(R)$.  
 \vspace{2mm}
	\begin{theorem}[{{\it Ideals in Excision Ring}}]\label{idealtheorem}
		The ideals in the excision ring $R\oplus I$ are of the form $J\oplus I_1$ for some ideals $J,I_1$ of $R$ such that $I_1\subseteq I$.
	\end{theorem}
	
	\begin{proof}
		Let $\mathfrak{I}$ be an ideal of $R\oplus I$. Define $J:=\{r \in R| (r,i)\in \mathfrak{I} \text{ for some  } i\in I \}$. Then $J$ is additively closed; for if $r_1,r_2 \in J$, there exist $i_1,i_2 \in I $  such that $(r_1,i_1),(r_2,i_2)\in \mathfrak{I} $ implies $(r_1+r_2,i_1+i_2) \in \mathfrak{I}$. Thus $r_1+r_2 \in J$. Now, let $s\in J$. Then there exists $j\in I$ such that $(s,j) \in \mathfrak{I}$ and for $r \in R,\; (r,0)\in R \oplus I $. Therefore $(r,0)\cdot(s,j)=(rs,rj)\in \mathfrak{I}$, which implies $rs\in J$. Thus $J$ is an ideal of $R$.
  
		\vspace{2mm}
		
		Similarly the set $I_1$ defined by $I_1:=\{i\in I| (r,i)\in \mathfrak{I} \text{ for some  } r \in R \}$ forms an ideal of $R$ with $I_1\subseteq I$.
	\end{proof}
	
	Thus we get a characterization for maximal ideals in $R\oplus I$ in terms of maximal ideals of $R$ as follows.

	\begin{corollary}  \label{Maximal ideal}
		
		Let $\mathfrak{m}$ be a proper ideal of the commutative ring $R$. Then $\mathfrak{m}$ is a maximal ideal in $R$ if and only if $\mathfrak{m} \oplus I $ is a maximal ideal in the excision ring $R\oplus I$.	
	\end{corollary}
	
	\begin{proof}
 
		If $\mathfrak{m}$ is an ideal of the ring $R$, then $\mathfrak{m}\oplus I$ is an ideal of $R\oplus I$. If $\mathfrak{m}\oplus I$ is not a maximal ideal of $R\oplus I$, then by Theorem \ref{idealtheorem}, there exists a proper ideal $J$ of $R$ such that $\mathfrak{m}\oplus I \subseteq J \oplus I \subseteq R \oplus I$.  Since $\mathfrak{m} \subseteq J$ and $\mathfrak{m}$ is a maximal ideal, we have $J=\mathfrak{m}$. Thus $\mathfrak{m} \oplus I$ is a maximal ideal in $R\oplus I$. Conversely, if $\mathfrak{m} \oplus I  $ is a maximal ideal of $R \oplus I$, then $\mathfrak{m}$ is a maximal ideal of $R$.  
  
	\end{proof}

	\begin{notation}
		Let $I$ be an ideal of $R$. Let $I_{\mathfrak{m}}$ denote $I \otimes R_{\mathfrak{m}}$ and $ M_\mathfrak{m}[X]$ denote $(Q \perp \mathbb{H}(P))\otimes R_{\mathfrak{m}}[X]$, where $R_{\mathfrak{m}}$ is the localization of $R$ with respect to the multiplicatively closed set $R\backslash \mathfrak{m}$.  

	\end{notation}

 From \cite{AmbilyRao2020}, we have the following lemma:
 
	\begin{lemma}[{\cite[Theorem~3.10]{AmbilyRao2020}, {\it Local-Global Principle}}]\label{ALG} 
     \noindent Let $(Q,q)$ be a quadratic $R$-space of rank $n\geq 1$ and $P$ be a finitely generated projective $R$ module of rank $m\geq 2$ and $M=Q\perp\mathbb{H}(P)$. Let $\alpha(X) \in {\rm O}_{R[X]}(M[X])$ with $\alpha(0) = Id$. Assume that for all maximal ideal $\mathfrak{m}$ of $R$, the  quadratic space $M_{\mathfrak{m}}$ is isomorphic to $R_{\mathfrak{m}}^{n+2m}$ with the matrix corresponding to the bilinear form is  $q_{\mathfrak{m}} \perp \widetilde{\phi}_{2m}.$ If for all maximal ideals $\mathfrak{m}$ of $R, \alpha(X)_{\mathfrak{m}} \in {\rm EO}_{R_{\mathfrak{m}}[X]}(Q_{\mathfrak{m}}[X],\mathbb{H}(R_{\mathfrak{m}}[X])^m)$, then $\alpha(X) \in
		{\rm EO}_{R[X]}(Q[X],\mathbb{H}(P[X]))$.
  
	\end{lemma}

Now we recall the definition of the spinor norm for an orthogonal group defined as follows:

	\begin{definition}[\it Spinor norm]
		Let $R$ be a local ring and let $Q$ be a quadratic $R$-space. Then the homomorphism $$N_s:{\rm O}_R(M)\to {\frac{R^*}{{R^*}^2}},$$ is defined by
  $$\tau_{v}\mapsto \langle v,v \rangle,$$ where $\tau_{v}$ is the hyperplane reflection given by $\tau_v(x)=x-2v\frac{\langle v,x \rangle}{\langle v,v \rangle }$ for $x\in Q$. 	
	\end{definition}
 \vspace{2mm}
       Now consider the following example on $R^2$.
 
	\begin{example}
		The matrix $V=\begin{pmatrix}
			u&0\\
			0&u^{-1}
		\end{pmatrix}$, where $u\in R$ can be expressed as the product $\tau_{e_1-e_2}\tau_{e_1-ue_2}$, where $e_1=(1,0)$ and $e_2=(0,1)$. Then the spinor norm of $V$ is $4u$.
	\end{example}

	\begin{definition}[{{\it Solvable group}}]
		Let $G$ be a group. Define the derived series $G=G_{(0)} \supseteq G_{(1)} \supseteq G_{(2)} \supseteq \cdots $, where $G_{(n+1)}=[G_{(n)},G_{(n)}]$. The group $G$ is said to be {\it solvable} if $G_{(n)}=\{e\}$ for some $n \in \mathbb{Z}$. The least value of $n$ for which $G_{(n)}=\{e\}$ is called the {\it length of the solvable group} $G$.
	\end{definition}

	\begin{definition}[{{\it Nilpotent group}}]
	Let $G$ be a group. Define the upper central sequence $G=G^{(0)} \supseteq G^{(1)} \supseteq G^{(2)} \supseteq \cdots $, where $G^{(n+1)}=[G,G^{(n)}]$. The group $G$ is said to be {\it nilpotent} if $G^{(n)}=\{e\}$ for some $n \in \mathbb{Z}$. The least value of $n$ for which $G^{(n)}=\{e\}$ is called {\it nilpotency class of the group} $G$.
\end{definition}

 Now we recall the following lemma about the decomposition of an orthogonal group of a quadratic $R$-space having sufficiently large Witt index.
	\begin{lemma}[{\rm \cite[Lemma~2.2]{RARao1984}}]\label{splitting}
		Let $(Q,q)$ be a diagonalizable quadratic $R$-space. Then for $m \geq \text{{\rm dim }{\rm Max}}(R) + 1$,
		$${\rm O}_{R}(Q\perp \mathbb{H}(R)^m)={\rm EO}_{R}(Q,\mathbb{H}(R)^m)\cdot{\rm O}_{R}(\mathbb{H}(R)^m)={\rm O}_{R}(\mathbb{H}(R)^m)\cdot{\rm EO}_{R}(Q,\mathbb{H}(R)^m).$$
	\end{lemma}
	
\section{Normality of relative DSER elementary orthogonal group}

Let $Q$ be a quadratic $R$-space and $P$ be a finitely generated projective $R$-module. Define $M=Q\perp \mathbb{H}(P)$. For $\alpha\in {\rm Hom}(M,M) $ and $\beta\in {\rm Hom}(M,IM)$, a transformation $(\alpha,\beta)$ on $M\oplus IM$ can be defined as follows: $$(\alpha,\beta)(u,v):=(\alpha(u),\alpha(v)+\beta(u)+\beta(v)).$$

\noindent{\bf Notation:} Let $R':=R \oplus I$, $I':= 0\oplus I$, $Q':=Q \oplus IQ$ and $P':= P\oplus IP$. Define $M':= Q'\perp \mathbb{H}(P')$. 

\vspace{2mm}

In \cite{Suresh1994}, V. Suresh proved the {\it splitting lemma} for the DSER elementary orthogonal group as follows:

\begin{lemma}[{\rm \cite[Lemma~1.2]{Suresh1994}}]
    Let $\alpha_1,\alpha_2\in {\rm Hom }(Q,P) ( {\rm or }\;  \beta_1,\beta_2\in {\rm Hom}(Q,P^*))$, then we have the splitting for the elementary transformations $E_{\alpha}$ $($and $E_{\beta}^*)$ as follows

    $$E_{\alpha_1+\alpha_2}=E_{\frac{\alpha_1}{2}} E_{\alpha_2} E_{\frac{\alpha_1}{2}}({\rm or \;} E_{\beta_1+\beta_2}^*=E_{\frac{\beta_1}{2}}^* E_{\beta_2}^* E_{\frac{\beta_1}{2}}^* ).$$
\end{lemma}
\vspace{2mm}

\begin{lemma}{\label{elementarylift}}
    Let $\epsilon \in {\rm EO}_{(R,I)}(Q,\mathbb{H}(P))$. Then there exists a lift $\epsilon'\in {\rm EO}_{(R',I')}(Q',\mathbb{H}(P'))$ of $\epsilon$ such that $f(\epsilon')=\epsilon$.
\end{lemma}

\begin{proof}

The generators of the relative elementary subgroup ${\rm EO}_{(R,I)}(Q,\mathbb{H}(P))$ are of the form $\epsilon_{\alpha,\beta}:=E'({\alpha})E'({\beta})E'({\alpha})^{-1}$, where $E'(-)$ is one of the elementary generators $E_-$ or $E^*_-$ of ${\rm EO}_{R}(Q,\mathbb{H}(P))$ with $\alpha \in {\rm Hom}(Q,P)$ or ${\rm Hom}(Q,P^*)$ and $\beta \in {\rm Hom}(Q,IP)$ or ${\rm Hom}(Q,IP^*)$. Now take the lift as $\widetilde {\epsilon}_{(\alpha,0), 
(0,\beta)}:= E'({(\alpha,0)}) E'({(0,\beta)}) E'({(\alpha,0)})^{-1} $, where $(\alpha,0)$ $ \in {\rm Hom}(Q',P')$ or ${\rm Hom}(Q',P'^*)$ and $(0,\beta)\in {\rm Hom}(Q',I'P')$ or ${\rm Hom}(Q',I'P'^* )$.

 \end{proof}

\begin{lemma}{\label{orthogonallift}}
      Let $\epsilon \in {\rm O}_{(R,I)}(M,IM)$. Then there exists a lift $\epsilon'\in {\rm O}_{(R', I')}(M',I'M')$ of $\epsilon$ such that $f(\epsilon')=\epsilon$.
\end{lemma}

 \begin{proof}
For $\epsilon \in {\rm O}_{(R,I)}(M,IM)$, define ${\epsilon'}:=(Id,\epsilon-Id).$ Since $\pi(\epsilon)=Id$, we have $\pi(\epsilon')=(Id,0)$. Clearly $\epsilon'$ is an orthogonal transformation on $M\oplus IM$ if $\epsilon$ is an orthogonal transformation on $M$. Thus $\epsilon \in {\rm O}_{(R,I)}(M,IM)$ implies $\epsilon' \in {\rm O}_{(R',I')}(M',IM')$.

 \end{proof}

	\begin{lemma}\label{steinbergrelns}
 Let $S$ be a commutative ring and $J$ be an ideal of $S$ such that a subring $R$ of $S$ maps isomorphically onto $S/J$. Let $Q$ be a quadratic $S$-space and $P$ be a projective $S$-module. Then the group $$G:=O_{(S,J)}(Q\perp \mathbb{H}(P))\cap {\rm EO}_{S}(Q,\mathbb{H}(P))$$ is generated by elements of the form $\gamma E'(\alpha)\gamma^{-1}$ where $\gamma\in {\rm EO}_{R}(Q,\mathbb{H}(P))$, $E'(-)$ is an elementary generator of ${\rm EO}_{S}(Q,\mathbb{H}(P))$ with $\alpha \in {\rm Hom}(Q,JP)$ or ${\rm Hom}(Q,JP^*)$. In particular, 
 $${\rm EO}_{(S,J)}(Q,\mathbb{H}(P))={\rm O}_{(S,J)}(Q\perp \mathbb{H}(P))\cap {\rm EO}_{S}(Q,\mathbb{H}(P)).$$
	\end{lemma}
\begin{proof}
    Let $\sigma \in G$. Then $\sigma =\prod_{k=1}^m E'(\alpha_{k})$ where $\alpha_{k} \in {\rm Hom}(Q,P)$ or ${\rm Hom}(Q,{P}^*)$. where $\alpha_k=\rho_k+\tau_k$ where $\rho_{k} \in {\rm Hom}_R(Q,P) $ or ${\rm Hom}_R(Q,P^*)$ and $\tau_{k} \in {\rm Hom}(Q,JP)$ or ${\rm Hom}(Q, JP^*)$. Now by splitting property, we have 
\begin{align*}
    \sigma &=\prod_{k=1}^{m}E'(\frac{\rho_k}{2})E'(\tau_k)E'(\frac{\rho_k}{2})\\&=\prod_{k=1}^{n+1}a_kb_k, 
\end{align*}

    \noindent where $a_1=E'(\frac{\rho_1}{2})$,
    $a_k=E'(\frac{\rho_{k-1}}{2})E'(\frac{\rho_{k}}{2})$  for  $k=2,\cdots,m$,
    $a_{m+1}=E'(\frac{\rho_{k}}{2})$, and $b_k=E'(\tau_k)$  for  $k=1,\cdots m \quad$ and $b_{m+1}=1$. 
Then we get $$\sigma =(\prod_{k=1}^{m}\gamma_kE'(\tau_k)\gamma_{k}^{-1})\gamma_{m+1},$$ 
where $\gamma_k=\prod_{l=1}^{k}a_k$. Since $\sigma\in {O}_{(S,J)}(Q,\mathbb{H}(P))$ and $S/J$ is isomorphic to $R$, we have $\gamma_{m+1}=Id$. Thus the result follows.
\end{proof}

\vspace{1mm}
From \cite{AmbilyRao2020}, we have the following theorem about the normality of the DSER elementary orthogonal group.   
 \begin{theorem}[{\rm \cite[Theorem~1.2]{AmbilyRao2020}}]\label{NormalityAbsolute} Let $Q$ be a quadratic $R$-space of rank $n$ and $P$ be a finitely generated projective $R$-module of rank $m$.	Then  for $n\geq 1$ and $m\geq 2$, the DSER elementary orthogonal group ${\rm EO}_{R}(Q, \mathbb{H}(P))$ is a normal subgroup of the orthogonal group ${\rm O}_R(Q\perp \mathbb{H}(P))$.
\end{theorem}

With the aid of Theorem \ref{NormalityAbsolute}, now we prove the normality of the relative DSER elementary orthogonal subgroup ${\rm EO}_{(R,I)}(Q,\mathbb{H}(P))$ in ${\rm O}_{R}(M)$.
\begin{theorem}\label{relativenormality}
     Let $Q$ be a quadratic $R$-space of rank $n\geq 1$ and $P$ be a projective module of rank $m\geq 2$. Then the DSER elementary orthogonal group ${\rm EO}_{(R,I)}(Q,\mathbb{H}(P))$ is a normal subgroup of ${\rm O}_{R}(Q\perp \mathbb{H}(P))$.
\end{theorem} 

    \begin{proof}
       Let $\alpha \in {\rm EO}_{(R,I)}(Q,\mathbb{H}(P))$ and $\beta \in {\rm O}_{R}(Q\perp \mathbb{H}(P))$. Using Lemma \ref{elementarylift} and Lemma \ref{orthogonallift}, we have $\widetilde{\alpha}\in {\rm EO}_{(R',I')}(Q',\mathbb{H}(P'))$ and $\widetilde{\beta} \in {\rm O}_{R'}(M')$ such that $f(\widetilde{\alpha})=\alpha$ and $f(\widetilde{\beta})=\beta$. Now, by applying the normality of the DSER elementary orthogonal group and ${\rm O}_{(R,I)}(Q\perp \mathbb{H}(P))$ in ${\rm O}_{R}(Q\perp \mathbb{H}(P))$ we have $\widetilde{\beta} \widetilde{\alpha} \widetilde{\beta}^{-1}\in {\rm O}_{(R',I')}(Q'\perp \mathbb{H}(P'))\cap {\rm EO}_{R'}(Q',\mathbb{H}(P'))$. Since $R'/I'$ is isomorphic to a subring $R$ of $R'$, Lemma \ref{steinbergrelns} implies that the product $\widetilde{\beta} \widetilde{\alpha} \widetilde{\beta}^{-1} \in {\rm EO}_{(R',I')}(Q',\mathbb{H}(P'))$. Hence, the projection $f(\widetilde{\beta}\widetilde{\alpha}\widetilde{\beta}^{-1})=\beta\alpha\beta^{-1} \in {\rm EO}_{(R,I)}(Q,\mathbb{H}(P))$.
    \end{proof}

\vspace{1mm}

\begin{remark}
    By definition of ${\rm EO}_{(R,I)}(Q,\mathbb{H}(P))$ we have, $$[{\rm EO}_{(R,I)}(Q,\mathbb{H}(P)),{\rm EO}_{R}(Q,\mathbb{H}(P))]={\rm EO}_{(R,I)}(Q,\mathbb{H}(P)). $$

    Using the above theorem, now we have the following commutator relation $$[{\rm EO}_{(R,I)}(Q,\mathbb{H}(P)),{\rm O}_{R}(Q \perp \mathbb{H}(P))]={\rm EO}_{(R,I)}(Q,\mathbb{H}(P)),$$ for a free module $P$. This formula is referred as {\it standard commutator formula} in the literature\cite{VavilovStepanov}.
\end{remark}
    
	\section{Relative Dilation Principle and local-global principle}
 
In this section, we aim to prove the dilation principle for the relative DSER elementary orthogonal group and the associated local-global principle.

Let $R$ be a commutative ring in which $2$ is invertible and let $I$ be an ideal of $R$.

\begin{lemma}\label{decomposition1}
The generators of ${\rm EO}_{(R[X],I[X])}(M[X]) \cap {\rm O}_{(R[X],(X))}(M[X],XM[X])$ are of the form 
$$\gamma E'(X\alpha'(X))\gamma^{-1},$$
where $\gamma \in {\rm EO}_{R}(M)$, $E'(X\alpha')=E_{X\alpha'}$ or $E_{X\alpha'}^*$  for $\alpha' \in {\rm Hom}(Q[X],IP[X]) $ or ${\rm Hom}(Q[X],IP^*[X])$.
\end{lemma}
\begin{proof}
    The generators of the group ${\rm EO}_{(R[X],I[X])}(M[X])$ are of the form  $E_{(\gamma,\alpha')}:=\gamma^{}{\rm E}(X\alpha'(X))\gamma^{-1}$ for $\alpha'(X)\in {\rm Hom}(Q[X],P[X])$ or ${\rm Hom}(Q[X],P^*[X])$ such that $E_{(\gamma,\alpha')}$ congruent to $Id$ modulo $I[X]$. Thus $Im(\alpha'(X))\subseteq IP[X]$ or $IP^*[X]$.
\end{proof}

\begin{lemma}[{\rm \cite[Corolary~3.7]{AmbilyRao2020}}]\label{decomposition}
     Let $Q$ be a quadratic $R$-space and $P$ be a finitely generated projective $R$-module. Let $s\in R$ be a non-nilpotent element such that $Q_s$ and $P_s$ are free modules of rank $n$ and $m$ respectively. Let $\gamma$ be the product of elementary generators in ${\rm EO}_{R_s}(Q_s,\mathbb{H}(P_s))$. Then for $x\in R$, $r\geq 0$, $ 1\leq k\leq m$ and $1\leq l\leq n$, there exist $x_t\in R$, $k_t,l_t\in \mathbb{Z}$ with $ 1\leq k_t \leq m$ and $1\leq l_t \leq n$ such that
$$\gamma E'(s^{{2^r}d}x M_{k,l})\gamma^{-1}=\prod_tE'(s^d x_tN_{k_t,l_t}),$$
where $M,N\in {\rm Hom}(Q_s,P_s)$ or ${\rm Hom}(Q_s,P_s)$ and $E'(-)$ denote the elementary generator $E_-$ or $E^*_-$ of ${\rm EO}_{R_s}(Q_s,\mathbb{H}(P_s))$.
\end{lemma}
\vspace{2mm}

Now we can prove the dilation principle for the relative DSER elementary orthogonal group.

Let $M=Q\perp \mathbb{H}(P)$ where $Q$ is a quadratic $R$-space and $P$ is a finitely generated projective module.

\begin{theorem}\label{dilation}
    Let $(Q,q)$ be a quadratic $R$-space of rank $n \geq 1$,  $P$ be a finitely generated projective $R$-module of rank $m$ and $M=Q\perp \mathbb{H}(P)$. Let $s\in R$ be a non-nilpotent element such that $Q_s$ and $P_s$ are free modules. Also, assume that the matrix corresponding to $Q_s$ is diagonal and the matrix corresponding to $\mathbb{H}(P_s)$ is $\widetilde{\phi}_{2m}.$ Let $\alpha(X) \in {\rm O}_{(R[X],I[X])}(M[X])$ with $\alpha(0)=Id$ and suppose $\alpha(X)_s\in {\rm EO}_{(R_s[X],I_s[X])}(M_s[X])$, then there exists $\hat{\alpha}(X) \in {\rm EO}_{(R[X],I[X])}(M[X])$ and $l>0$ such that $\hat{\alpha}(X)$ localizes to $\alpha(bX)$ for some $b\in (s^l)$ and $\hat{\alpha}(0)=Id$.
\end{theorem}
    \begin{proof}
        Given that $\alpha(X) \in {\rm EO}_{R[X]}(Q[X],\mathbb{H}(P[X]))$ with $\alpha(0)=Id$. For a non-nilpotent element $s\in R$, we have, $\alpha_s(X) \in {\rm EO}_{(R_s[X],I_s[X])}(Q_s[X],\mathbb{H}(P_s[X]))$. Thus by applying Lemma \ref{decomposition1} we have 
        
        $$\alpha(X)_s =\prod_k \gamma_k E'(X\alpha_k'(X))\gamma_k^{-1}$$
        
        \noindent where, $\alpha'(X) \in {\rm Hom}(Q_s[X],I_sP_s[X])$ or ${\rm Hom}(Q_s[X],I_sP_s^*[X])$ and $\gamma_k \in {\rm EO}_{R}(M)$. By Lemma \ref{elementarylift}, there are lifts $\widetilde{\alpha}(X)_{(s,0)} \in {\rm EO}_{(R'_s,I'_s)}(M'_s[X])$ and $\widetilde{\gamma_k} \in {\rm EO}_{R'}(M')$ such that 
        
        $$ \widetilde{\alpha}(X)_{(s,0)}:=\prod_k \widetilde{\gamma_k} \widetilde{E}'(0,X\alpha_k'(X))\widetilde{\gamma_k}^{-1},$$ 
        
        \noindent where $(0,X\alpha'(X)) \in {\rm Hom}(Q'_s[X],I'P'_s[X])$ or ${\rm Hom}(Q'_s[X],I'P'_s[X])$. 
        
        \vspace{2mm}

        \noindent Using the commutator identities for the generators of the DSER elementary orthogonal group in free case (for details see \cite[Section~3]{Ambily2019}), for $l>0$, we have $\widetilde{\alpha}(Y^{2l}X)$ is the product of elementary generators  $E(0,Y^{l}\alpha_{i_t,j_t}(X,Y)/s^{\mu_t})$ and the commutators of the elementary generators of the form $[E'(0, Y^{l}\alpha'_{i_t,j_t}(X,Y)/s^{\mu_t}),E'(-)]$, or $[E'(-),E'(0, Y^{l}\alpha''_{i_t,j_t}(X,Y)/s^{\mu_t})]$  where, $(0,Y^l\alpha_{i_t,j_t}(X,Y)),(0,\alpha'_{i_t,j_t}(X,Y)), (0,\alpha''_{i_t,j_t}(X,Y)) \in {\rm Hom}(Q'[X,Y],I'P'[X,Y]) $.

 \vspace{2mm}

        \noindent As a consequence of Lemma \ref{decomposition}, we have $\alpha(s^{2^rkl}Y^{2^rkl}X)$ is product of $E'((0,s^lY^lW_{i_t,j_t}(X,Y)))$ and the commutators of the form $[E'(0,s^lY^lW'_{i_t,j_t})(X,Y),E'(-)]$,$[E'(-),E'(0,s^lY^lW''_{i_t,j_t})(X,Y)]$ for $W_{i_t,j_t}(X,Y), W'_{i_t,j_t}(X,Y)\in {\rm Hom}(Q[X,Y],IP[X,Y])$ for $r$ large enough.
        
         \vspace{3mm}
        
        Thus we have  
        
        $$\alpha(s^{2^rkl}Y^{2^rkl}X) \in {\rm EO}_{(R'[X,Y],I'[X,Y])}(Q'[X,Y],\mathbb{H}(P'[X,Y])).$$
        
        %$$\alpha(s^{2^rkl}Y^{2^rkl}X)=\prod_{k=1}^{\nu}\prod_{t=1}^{\nu_{r_k}}E(s^lY^lW_{i_t,j_t}(X,Y)) \in {\rm EO}_{(R'[X,Y],I'[X,Y])}(Q'[X,Y],\mathbb{H}(P'[X,Y])).$$

        \noindent By \cite[Lemma~3.8]{AmbilyRao2020}, we have $\hat{\widetilde{\alpha}}_{(s,0)}(X,Y)={\widetilde{\alpha}}_{(s,0)}((b,0)XY^{2d})$. Thus, for $Y=(1,0)$, we get $\hat{\widetilde{\alpha}}_{(s,0)}((b,0)XY^{2d}) \in {\rm EO}_{(R',I')}(M'[X],I'M'[X])$. Now by applying projection, we have $\alpha(bX)\in {\rm EO}_{(R,I)}(Q,\mathbb{H}(P))$. 
        \end{proof}
        \vspace{2mm}
        
The local-global principle for the relative DSER elementary orthogonal group can be obtained from the relative dilation principle as follows:

        \begin{theorem}[Relative local-global principle]\label{RLG}
           Let $Q$ be a quadratic $R$-space of rank $n\geq 1$ and  $P$ be a finitely generated projective module of rank $m\geq 2$ and $M=Q\perp \mathbb{H}(P)$ and let $\alpha(X) \in {\rm O}_{R[X]}(M[X])$. If for all maximal ideal $\mathfrak{m}$ of $R$, $\alpha_{\mathfrak{m}}(X) \in {\rm EO}_{(R_\mathfrak{m}[X],I_{\mathfrak{m}}[X])}(Q_{\mathfrak{m}}[X],\mathbb{H}(P_{\mathfrak{m}}[X]))$, with $\alpha(0)=Id$, then $\alpha(X) \in {\rm EO}_{(R[X],I[X])}(Q[X],\mathbb{H}(P[X]))$.
        \end{theorem}

        \begin{proof}
        Let $\mathfrak{m}$ be a maximal ideal of the ring $R$. Choose $s_{\mathfrak{m}} \in R\backslash\mathfrak{m}$ such that $$\alpha_{s_{\mathfrak{m}}}(X)\in {\rm EO}_{(R_{s_{\mathfrak{m}}}[X], I_{s_{\mathfrak{m}}}[X])}(Q_{s_{\mathfrak{m}}}[X],\mathbb{H}(P_{s_{\mathfrak{m}}})[X]).$$

        Define $$\beta(X,Y):=\alpha(X+Y)\alpha(Y)^{-1}\in {\rm EO}_{R[X,Y]}(M[X,Y]).$$

        Then clearly $\beta(0,Y)=Id$ and $\beta(X,Y)_{s_{\mathfrak{m}}} \in {\rm EO}_{(R_{s_{\mathfrak{m}}}[X,Y],I_{s_{\mathfrak{m}}}[X,Y])}(M_{s_{\mathfrak{m}}}[X,Y])$. By Theorem \ref{dilation}, we have $d>>0$ such that $\beta(b_{\mathfrak{m}}X,Y) \in {\rm EO}_{(R[X,Y],I[X,Y])}(M[X,Y])$ for $b_{\mathfrak{m}}\in (s_{\mathfrak{m}}^d)$. Since the elements of the form $s_{\mathfrak{m}}^d$ generate the ring $R$, there are elements $b_{\mathfrak{m}_i}\in R$ such that $\sum_{i=1}^r b_{\mathfrak{m}_i} = 1$. Thus $\alpha(X)=\prod_{i=1}^{r-1}\beta(b_{\mathfrak{m}_i}X,\sum_{k=i+1}^rb_{\mathfrak{m}_k}X)\beta(b_{\mathfrak{m}_r}X,0) \in {\rm EO}_{(R[X],I[X])}(M[X],IM[X])$.
        \end{proof}

\section{Equality of DSER elementary orthogonal group and usual elementary orthogonal group in free case}
	
In  \cite{AmbilyRao2020}, A. A. Ambily and R. A. Rao proved the equality of elementary orthogonal group and DSER elementary orthogonal group on the orthogonal sum of a quadratic $R$-space of even dimension and hyperbolic plane as follows.
 
 \begin{lemma}[{\rm \cite[Lemma~2.7]{AmbilyRao2020}}]
Let $Q$ be a free quadratic space of rank $2n$ over $R$, with $n\geq 1$ and the matrix corresponding to the bilinear form be $\tilde{\phi}_{2n}$. Then
$${\rm EO}_R(Q,\mathbb{H}(R)^m)={\rm EO}_{2(n+m)}(R).$$
 \end{lemma}
 
 \noindent Similar result for an odd dimensional quadratic space can be proved as
 \begin{lemma}
     Let $Q$ be a quadratic space of rank $1$. Then $$  {\rm EO}_{R}(Q,\mathbb{H}(R)^m)={\rm EO}_{2m+1}(R).$$
 \end{lemma}

 \begin{proof}
     For $\lambda \in R$, we have the relations  $F_i^1(\lambda) = E_{{\alpha}_{i1}} (-2\lambda)$ and $F_i^2(\lambda) = E^*_{{\alpha}_{i1}} (-2\lambda)$ for the generators $F_i^1$ and $F_i^2$ of  ${\rm EO_{2n+1}(R)}$, where $\alpha_{ij}(\lambda)=\lambda e_{ij} \in {\rm Hom}(Q,P)$ or ${\rm Hom}(Q,P^*)$. Using the commutator identities (see \cite[Section~3]{Ambily2019}) and the above equalities, we have ${\rm EO_{2m+1}(R)} = {\rm EO_{R}(Q, \mathbb{H}(R)^m)}.$

 \end{proof}

\noindent Using \cite[Lemma~3.6]{AARR2023},   and by appending both of the above lemmas, we have the following lemma.

	\begin{lemma} \label{freecase}
		If $(Q,\tilde{\phi}_n)$ is a free quadratic $R$-space of rank $n\geq 1$. Then $${\rm EO}_{R}(Q ,\mathbb{H}(R)^m) = {\rm EO}_{n+2m}(R).$$ In particular, for any ideal $I$ of $R$, the group ${\rm EO}_{(R,I)}(Q ,\mathbb{H}(R)^m) = {\rm EO}_{n+2m}(R,I)$.
	\end{lemma}
	
	\begin{remark}
		Let $R$ be a ring and $I$ be an ideal of $R$. Then by Theorem \ref{NormalityAbsolute}, ${\rm EO}_{2m+1}(R)$ is a normal subgroup of ${\rm SO}_{2m+1}(R)$ for $m\geq 2$. Moreover, using Theorem \ref{relativenormality}, we have ${\rm EO}_{2m+1}(R,I)$ is a normal subgroup of ${\rm SO}_{2m+1}(R,I)$ for $m\geq 3$.
	\end{remark}

	\section{${\rm SK_1}$-analog of the relative DSER elementary orthogonal group}

        In this section, we assume the hyperbolic rank $m\geq 2$ unless it is specified.

              \vspace{2mm}

      In \cite{RRaoSS2017}, R. A. Rao and S. Sharma studied the solvability of $\frac{{\rm SO}_{2m}(R[X])}{{\rm EO}_{2m}(R[X])}$, when $R$ is a local ring. They also studied the solvability of corresponding relative quotient groups. By analogous method, we extend these results to odd dimensional orthogonal groups to study the ${\rm SK_1}$-analog of the relative DSER elementary orthogonal group and the relative transvection groups.
      
      \vspace{2mm}

      To begin with, we shall study certain properties of the relative special orthogonal group ${\rm SO}_{2m+1}(R,I)$ and the quotient group $\frac{{\rm SO}_{2m+1}(R,I)}{{\rm EO}_{2m+1}(R,I)}$ over a local ring $R$. 
      
	\begin{theorem}\label{squareinelementary}
		Let $R$ be a local ring and $I$ be an ideal of $R$. Let $\alpha \in {\rm SO}_{2m+1}(R,I)$. Then $\alpha^2 \in {\rm EO}_{2m+1}(R,I)$.  
	\end{theorem}
	
	\begin{proof}
 
		Let $\alpha \in {\rm SO}_{2m+1}(R,I) \subseteq {\rm O}_{2m+1}(R,I)={\rm Ker}({\rm O}_{2m+1}(R) \rightarrow {\rm O}_{2m+1}(R/I))$. Then there exists $\alpha' \in M_{2m+1}(I) $ such that $\alpha=Id+ \alpha'$. Now,  consider $\widetilde{\alpha}=(Id, \alpha') \in M_{2m+1}(R\oplus I)$. Since $\alpha^t \tilde{\phi}_{2m+1} \alpha = (Id+\alpha')^t\tilde{\phi}_{2m+1}(Id+\alpha') \in {\rm SO}_{2m+1}(R)$, we have $$\widetilde{\alpha}^t \tilde{\phi}_{2m+1} \widetilde{\alpha} = (Id, \alpha')^t\tilde{\phi}_{2m+1}(Id, \alpha')=\tilde{\phi}_{2m+1}.$$ Consequently, we get $ \widetilde{\alpha} \in  {\rm SO}_{2m+1}(R\oplus I,0\oplus I)$.
		
		\vspace{2mm}
		
		Considering the Lemma \ref{splitting}, for $n=1$, we have ${\rm SO}_{2m+1}(R\oplus I)={\rm SO}_2(R \oplus I)\cdot{\rm EO}_{2m+1}(R\oplus I)$ and by Corollary \ref{Maximal ideal}, $R\oplus I $ is a local ring. In view of \cite[Lemma~2.2]{AmbilyRao2020}, every element of ${\rm SO}_2(R\oplus I)$ is of the form $$\alpha_s=\begin{pmatrix}
			(r,i)&0\\0&(r^{-1},j)\\
		\end{pmatrix},$$ for some $j$ with $rj+r^{-1}i+ij=0$, where $r^{-1}$ is the inverse of $r$. Thus every element $\widetilde{\alpha} \in {\rm SO}_{2m+1}(R\oplus I) $ is of the form $$\widetilde{\alpha}=\alpha_s\alpha_e,$$ for some $\alpha_s \in {\rm SO}_{2}(R\oplus I)$ and $\alpha_e \in {\rm EO}_{2m+1}(R\oplus I)$. Then by the normality of the elementary orthogonal group ${\rm EO}_{2m+1}(R\oplus I) $ in ${\rm SO}_{2m+1}(R\oplus I)$, we have  $$\widetilde{\alpha}^2=\alpha_s^2\alpha_e',$$ for some $\alpha_e' \in {\rm EO}_{2m+1}(R\oplus I)$. Then the spinor norm of $$\alpha_s^2=\begin{pmatrix}
			(r,i)^2&0\\0&(r^{-1},j)^2\\
		\end{pmatrix}$$  is $4(r,i)^2$ which is a square in $(R\oplus I)^*$. On account of \cite[Theorem~4]{Klingenberg1961}, we have $\alpha_s^2 \in {\rm EO}_{2}(R\oplus I)$. Thus $\widetilde{\alpha}^2 \in {\rm EO}_{2m+1}(R\oplus I)$. As a consequence of Lemma \ref{steinbergrelns}, $\widetilde{\alpha}^2 \in {\rm EO}_{2m+1}(R\oplus I,0\oplus I)$.
		
		\vspace{2mm}
		
		Thus $\widetilde{\alpha}^2 = \prod_{}^{}F^{i_n}_{j_n}(r_n,t_n)F^{k_n}_{l_n}(0,s_n){F^{i_n}_{j_n}}^{-1}(r_n,t_n)$, where $ F^{i_n}_{j_n}(r_n,t_n) $ and $F^{k_n}_{l_n}(0,s_n)$ are some elementary generators of ${\rm EO}_{2m+1}(R\oplus I)$ and ${\rm EO}_{2m+1}(0 \oplus I)$ for $r_n,t_n\in R$ and $s_n \in I$. Taking the projection of $\tilde{\alpha}^2$  onto ${\rm EO}_{2m+1}(R,I)$, we get $$\alpha^2=\prod_{}^{}F^{i_n}_{j_n}(r_n+t_n)F^{k_n}_{l_n}(s_n){F^{i_n}_{j_n}}^{-1}(r_n+t_n) \in {\rm EO}_{2m+1}(R,I).$$
  
	\end{proof}
	
	\begin{corollary}\label{Abeliangrouprelative}
		Let $R$ be a local ring and $I$ be an ideal of $R$. Then every element of the group $\frac{{\rm SO}_{2m+1}(R,I)}{{\rm EO}_{2m+1}(R,I)}$ has order 2.
	\end{corollary}
	\begin{lemma} \label{Abeliangroup}
		Let $R$ be a local ring. Then $\frac{{\rm SO}_{2m+1}(R)}{{\rm EO}_{2m+1}(R)} $ is abelian.
	\end{lemma}
 
	\begin{proof}
 
		Let $\widetilde{\alpha},\widetilde{\beta} \in\frac{{\rm SO}_{2m+1}(R)}{{\rm EO}_{2m+1}(R)}$. Then by Lemma \ref{splitting}, there exist $\alpha',\beta' \in {\rm SO}_2(R)$  such that $\alpha'$ and $\beta'$ are two coset representatives of  $\widetilde{\alpha},\widetilde{\beta}$ in $\frac{{\rm SO}_{2m+1}(R)}{{\rm EO}_{2m+1}(R)}$. Then, by the normality of ${\rm EO}_{2m+1}(R)$ in ${\rm SO}_{2m+1}(R)$ and the commutativity of ${\rm SO}_{2}(R)$, the elements $\widetilde{\alpha}$ and $\widetilde{\beta}$ commutes.
  
	\end{proof}

Let $\alpha, \beta \in {\rm O}_n(R)$. The matrices $\alpha$ and $\beta$ are said to be elementarily equivalent if $\alpha\beta^{-1} \in {\rm EO}_n(R)$.
 
	\begin{lemma}
		Let $R$ be a local ring and $\alpha(X),\beta(X) \in [{\rm SO}_{2m+1}(R[X]),{\rm SO}_{2m+1}(R[X])]$. Then the commutator $[\alpha(X),\beta(X)]$ is elementarily equivalent to $[\alpha(X)\alpha(0)^{-1},\beta(X)\beta(0)^{-1}]$ for $m\geq 3$.
	\end{lemma}
	
	\begin{proof}
 
For $\alpha(X),\beta(X) \in [{\rm SO}_{2m+1}(R[X]),{\rm SO}_{2m+1}(R[X])]$, we have
		\begin{align*}
			&[\alpha(X)\alpha(0)^{-1},\ \beta(X)\beta(0)^{-1}]\\
   &= \alpha(X)\alpha(0)^{-1}\beta(X)\beta(0)^{-1}\alpha(0)\alpha(X)^{-1}\beta(0)\beta(X)^{-1}\\
			&=(\alpha(X)\beta(X)\alpha(X)^{-1}\beta(X)^{-1})(\beta(X)\alpha(X)\beta(X)^{-1}\alpha(0)^{-1}\beta(X)\alpha(X)^{-1}\beta(X)^{-1})\\
			&\quad\ (\beta(X)\alpha(X)\beta(0)^{-1}\alpha(0)\alpha(X)^{-1}\beta(X)^{-1})\beta(X)\beta(0)\beta(X)^{-1}.
          \end{align*}

          \vspace{2mm}

\noindent where $\alpha(0),\beta(0) \in [{\rm SO}_{2m+1}(R),{\rm SO}_{2m+1}(R)]\subseteq {\rm EO}_{2m+1}(R)\subseteq {\rm EO}_{2m+1}(R[X])$. For $ m\geq 3$, we have ${\rm EO}_{2m+1}(R[X])$ is a normal subgroup of ${\rm SO}_{2m+1}(R[X])$. Consequently, the commutator $$[\alpha(X)\alpha(0)^{-1},\ \beta(X)\beta(0)^{-1}] \in [\alpha(X),\beta(X)]{\rm EO}_{2m+1}(R[X]).$$\end{proof}
\vspace{1mm}
In this case, we say that the commutator $[\alpha(X)\alpha(0)^{-1},\ \beta(X)\beta(0)^{-1}]$ is elementarily equivalent to the commutator $[\alpha(X),\beta(X)]$.

\vspace{.2cm}

	Now we can prove the solvability of ${\rm SK}_1$ of the usual elementary orthogonal group and the corresponding relative elementary group over the polynomial ring over a local ring. The local-global principle for elementary orthogonal group plays the key role in proving the following results.

	\begin{theorem}\label{SK_1analogfree}
		Let $R$ be a local ring. Then the quotient group $\frac{{\rm SO}_{2m+1}(R[X])}{{\rm EO}_{2m+1}(R[X])}$ is solvable of length at most 2.
	\end{theorem}

	\begin{proof}
		Let $\alpha(X),\beta(X) \in [{\rm SO}_{2m+1}(R[X]),{\rm SO}_{2m+1}(R[X])]$. To prove the length of solvability is 2, we need to prove $[\alpha(X),\beta(X)] \in {\rm EO}_{2m+1}(R[X])$. Without loss of generality, suppose $\alpha(0)=\beta(0)=Id$. (If not, since $[\alpha(X)\alpha(0)^{-1},\beta(X)\beta(0)^{-1}]$ is elementarily equivalent to $[\alpha(X),\beta(X)]$, the result follows.) Define $\gamma(X,T):= [\alpha(TX),\beta(X)]$. Then  $\gamma(X,0)=Id$ and for any maximal ideal $\mathfrak{m}$ of $R[X]$, we have $\gamma(X,T)_{\mathfrak{m}} = [\alpha(X,T)_{\mathfrak{m}}, \beta(X)_{\mathfrak{m}}]$. Furthermore, we have $ \beta(X)_{\mathfrak{m}}\in [{\rm SO}_{2m+1}(R[X]_{\mathfrak{m}}),{\rm SO}_{2m+1}(R[X]_{\mathfrak{m}})]\subseteq {\rm EO}_{2m+1}(R[X]_{\mathfrak{m}})$. By the normality of the elementary subgroup ${\rm EO}_{2m+1}(R[X]_{\mathfrak{m}}[T])$ in ${\rm SO}_{2m+1}(R[X]_{\mathfrak{m}}[T])$, we have $ \gamma(X,T)_{\mathfrak{m}} \in {\rm EO}_{2m+1}(R[X]_{\mathfrak{m}}[T])$. Then by Lemma \ref{ALG}, $\gamma(X,T)\in {\rm EO}_{2m+1}(R[X,T])$.  For $T=1,$ we get $\gamma(X,1)=[\alpha(X),\beta(X)] \in {\rm EO}_{2m+1}(R[X]).$
		
	\end{proof}
	\begin{theorem} \label{relativesolvabilityodd}
		Let $R$ be a local ring. Then the quotient group $\frac{{\rm SO}_{2m+1}(R[X],I[X])}{{\rm EO}_{2m+1}(R[X],I[X])}$ is solvable of length at most 2.
	\end{theorem}
 
	\begin{proof}
 
		Let $\alpha(X),\beta(X) \in [{\rm SO}_{2m+1}(R[X],I[X]),{\rm SO}_{m+1}(R[X],I[X])]$. To prove that the length of solvability is 2, it is sufficient to show that $[\alpha(X),\beta(X)] \in {\rm EO}_{2m+1}(R[X],I[X])$. Without loss of generality, suppose $\alpha(0)=\beta(0)=Id$.  Now define $\gamma(X,T):= [\alpha(XT),\beta(X)]$. Then  $\gamma(X,0)=Id$ and for any  maximal ideal $\mathfrak{m}$ of $R[X]$ we have $\gamma(X,T)_{\mathfrak{m}} = [\alpha(X,T)_{\mathfrak{m}}, \beta(X)_{\mathfrak{m}}]$. Moreover, we have  $ \beta(X)_{\mathfrak{m}}\in [{\rm SO}_{2m+1}(R[X]_{\mathfrak{m}},\; I[X]_{\mathfrak{m}}),{\rm SO}_{2m+1}(R[X]_{\mathfrak{m}},I[X]_{\mathfrak{m}})] \subseteq {\rm EO}_{2m+1}(R[X]_{\mathfrak{m}},I[X]_{\mathfrak{m}})$. Thus, by the  normality of relative elementary orthogonal group ${\rm EO}_{2m+1}(R[X]_{\mathfrak{m}}(T),I[X]_{\mathfrak{m}}[T])$ in ${\rm SO}_{2m+1}(R[X]_{\mathfrak{m}}[T],I[X]_{\mathfrak{m}}[T])$, we get $ \gamma(X,T)_{\mathfrak{m}} \in {\rm EO}_{2m+1}(R[X]_{\mathfrak{m}}[T],I[X]_{\mathfrak{m}}[T])$. By Lemma \ref{RLG}, the commutator $\gamma(X,T)\in {\rm EO}_{2m+1}(R[X,T],I[X,T])$.  Then, by putting $T=1$, $\gamma(X,1)=[\alpha(X),\beta(X)] \in \; {\rm EO}_{2m+1}(R[X],I[X])$.
  
	\end{proof}

    \vspace{2mm}
    
    Now we generalize these results to prove the solvability of ${\rm SK_1}$-analog of the DSER elementary orthogonal group.

    \vspace{2mm}

	{\bf Proof of theorem \ref{DSERSolvability}.}
	
	\begin{proof}
 
		Let $\alpha(X), \beta(X) \in [{\rm SO}_{(R[X],(X))}(M[X],XM[X]),{\rm SO}_{(R[X],(X))}(M[X],XM[X])]$ and define the commutator $\gamma(X):=[\alpha(X), \beta(X)]$. To prove the length of solvability of $\frac{{\rm SO}_{(R[X],(X))}(M[X],XM[X])}{{\rm EO}_{(R[X],(X))}(Q[X],\mathbb{H}(R[X])^m)}$ is at most 2, it is sufficient to prove that the commutator $\gamma(X)\in {\rm EO}_{(R[X],(X))}(Q[X],\mathbb{H}(R[X])^m)$. Let $\mathfrak{m} $ be any maximal ideal of $R$. Since finitely generated projective modules are locally free, after localization, we have  $\alpha(X)_{\mathfrak{m}}, \beta(X)_{\mathfrak{m}} \in [{\rm SO}_{n+2m}(R_{\mathfrak{m}}[X],(X)_{\mathfrak{m}}),{\rm SO}_{n+2m}(R_{\mathfrak{m}}[X],(X)_{\mathfrak{m}})]$. Then by Theorem \ref{relativesolvabilityodd} and \cite[Theorem~4.12]{RRaoSS2017},  we get that after localization with respect to the maximal ideal $\mathfrak{M}$ of $R$, $\gamma(X)_{\mathfrak{m}} \in {\rm EO}_{n+2}(R_{\mathfrak{m}}[X],(X)_{\mathfrak{m}})={\rm EO}_{(R_{\mathfrak{m}}[X],(X)_{\mathfrak{m}})}(M_{\mathfrak{m}}[X])$. Now by Theorem \ref{RLG}, we obtain $\gamma(X) \in {\rm EO}_{(R[X],(X))}(Q[X],\mathbb{H}(R[X])^m)$.
  
	\end{proof}
	
	\begin{theorem}
		Let $(Q,q)$ be a quadratic $R$-space of rank $n \geq 1$ and $P$ be a finitely generated projective module of rank $m\geq2$. Let $I$ be an ideal of $R$ and $M=Q\perp \mathbb{H}(P)$. If for all maximal ideal $\mathfrak{m}$ of $R,\; q_{\mathfrak{m}}\perp h_{\mathfrak{m}}=\tilde{\phi}_{n+2m}$, where $h$ is the hyperbolic form on $\mathbb{H}(P)$, then the group $\frac{{\rm SO}_{(R[X],XI[X])}(M[X],XIM[X])}{{\rm EO}_{(R[X],XI[X])}(Q[X],\mathbb{H}(R[X])^m)}$ is solvable of length at most $2$.
	\end{theorem}
	
	\begin{proof}
 
		Let $\alpha(X), \beta(X) \in [{\rm SO}_{(R[X],XI[X])}(M[X],XIM[X]),{\rm SO}_{(R[X],XI[X])}(M[X],XIM[X])]$. To prove that the quotient group $\frac{{\rm SO}_{(R[X],XI[X])}(M[X],XIM[X])}{{\rm EO}_{(R[X],XI[X])}(Q[X],\mathbb{H}(R[X])^m)}$ is solvable of length at most 2, it is sufficient to prove that the commutator $\gamma(X):=[\alpha(X), \beta(X)] \in {\rm EO}_{(R[X],XI[X])}(Q[X],\mathbb{H}(R[X])^m)$. Let $\mathfrak{m} $ be any maximal ideal of $R$. Then the elements $\alpha(X)_{\mathfrak{m}}, \beta(X)_{\mathfrak{m}} \in [{\rm SO}_{n+2m}(R_{\mathfrak{m}}[X]),{\rm SO}_{n+2m}(R_{\mathfrak{m}}[X])]$. Now by Theorem \ref{relativesolvabilityodd} and \cite[Theorem~4.12]{RRaoSS2017}, we have  $\gamma(X)_{\mathfrak{m}} \in {\rm EO}_{n+2}(R_{\mathfrak{m}}[X],XI[X]_{\mathfrak{m}})={\rm EO}_{(R_{\mathfrak{m}}[X],I_{\mathfrak{m}}[X])}(M_{\mathfrak{m}}[X])$. By applying Theorem \ref{RLG}, we get $\gamma(X) \in {\rm EO}_{(R[X],XI[X])}(Q[X],\mathbb{H}(R[X])^m)$.
  
	\end{proof}

\begin{lemma}
	Let $R$ be a local ring. Let $m\geq 2$ and let $\alpha(X),\beta(X) \in {\rm SO}_{2m+1}(R[X])$ be such that $\alpha(0)=Id$. Then the commutator $[\alpha(X),\beta(X)^2] \in {\rm EO}_{2m+1}(R[X])$. 
\end{lemma}

\begin{proof}

	Define $\gamma(X,T):=[\alpha(TX),\beta(X)^2]$. Then clearly $\gamma(X,0)=Id$. Then for any maximal ideal $\mathfrak{m}$ of $R[X],\; \gamma(X,T)_{\mathfrak{m}}=[\alpha(TX)_{\mathfrak{m}},\beta(X)_{\mathfrak{m}}]$, whereas $\beta(X)_{\mathfrak{m}} \in {\rm SO}_{2m+1}(R[X]_{\mathfrak{m}})$. Then by Lemma \ref{squareinelementary}, $\beta(X)_{\mathfrak{m}}^2 \in {\rm EO}_{2m+1}(R[X]_{\mathfrak{m}})  \subseteq {\rm EO}_{2m+1}(R[X]_{\mathfrak{m}}[T])$. Now by normality of the elementary subgroup ${\rm EO}_{2m+1}(R[X]_{\mathfrak{m}}[T]) $ in ${\rm O}_{2m+1}(R[X]_{\mathfrak{m}}[T])$, we have $ \gamma(X,T)_{\mathfrak{m}} \in {\rm EO}_{2m+1}(R[X]_{\mathfrak{m}}[T])$. Now by Lemma \ref{ALG}, we have $\gamma(X,T) \in {\rm EO}_{2m+1}(R[X,T])$. In particular, for $T=1$, we get $\gamma(X,1)=[\alpha(X),\beta(X)^2] \in {\rm EO}_{2m+1}(R[X])$.
 
\end{proof}

Next, we study the nilpotency of the quotient group $\frac{{\rm SO}_{2m+1}(R[X],(X))}{{\rm EO}_{2m+1}(R[X],(X))}$.
\begin{theorem}\label{nilpotency}
	Let $R$ be a local ring and $m\geq 2$. Then the nilpotency class of the quotient group $\frac{{\rm SO}_{2m+1}(R[X],(X))}{{\rm EO}_{2m+1}(R[X],(X))}$ is at most 2.
\end{theorem}
\begin{proof}

Let $\alpha(X)\in {\rm SO}_{2m+1}(R[X],(X))$. Then there exists $\delta(X)\in {\rm M}_{2m+1}(R[X])$ such that $\alpha(X)= Id +X\delta(X)$ . Therefore $\alpha(0)=Id$. For $\beta(X) \in [{\rm SO}_{2m+1}(R[X],(X)),{\rm SO}_{2m+1}(R[X],(X))]$, define the commutator $\gamma(X,T):=[\alpha(XT),\beta(X)]$. Clearly $\gamma(X,0)=Id$. Then for any maximal ideal $\mathfrak{m}$ of $R[X]$, we have $ \gamma(X,T)_{\mathfrak{m}}=[\alpha(XT)_{\mathfrak{m}},\beta(X)_{\mathfrak{m}}]$. But after localization, $\beta(X)_{\mathfrak{m}}\in [{\rm SO}_{2m+1}(R[X]_{\mathfrak{m}},(X)_{\mathfrak{m}}),{\rm SO}_{2m+1}(R[X]_{\mathfrak{m}},(X)_{\mathfrak{m}})]\subseteq {\rm EO}_{2m+1}(R[X]_{\mathfrak{m}}) \subseteq{\rm EO}_{2m+1}(R[X]_{\mathfrak{m}}[T])$. By normality of ${\rm EO}_{2m+1}(R[X]_{\mathfrak{m}}[T])$ in  ${\rm SO}_{2m+1}(R[X]_{\mathfrak{m}}[T])$ from Theorem \ref{relativenormality}, we have $\gamma(X,T)_{\mathfrak{m}} \in {\rm EO}_{2m+1}(R[X]_{\mathfrak{m}}[T])$. Now by Lemma \ref{RLG}, we get $\gamma(X,T) \in {\rm EO}_{2m+1}(R[X,T])$. In particular for $T=1,\;\gamma(X,1)=[\alpha(X),\beta(X)] \in {\rm EO}_{2m+1}(R[X]).$	Since $\alpha(0)=Id$, we have $\gamma(0,1)=[\alpha(0),\beta(0)]=Id$. Therefore, we get $[\alpha(X),\beta(X)] \in {\rm EO}_{2m+1}(R[X],(X))$. 

\end{proof}

Now we generalize the theorem to prove that the nilpotency class of the quotient group $\frac{{\rm SO}_{(R[X],(X))}(M[X],XM[X])}{{\rm EO}_{(R[X],(X))}(Q[X],\mathbb{H}(P[X]))}$ is at most 2.

\begin{theorem}\label{DSERnilpotencycase1}
	Let $(Q,q)$ be a quadratic $R$-space of rank $n\geq 1.$ Let $M=Q\perp \mathbb{H}(R)^m$ for $m\geq 2$. If for all maximal ideal $\mathfrak{m}$ of $R,\; q_{\mathfrak{m}}\perp h_{\mathfrak{m}}=\tilde{\phi}_{n+2m}$, where $h$ is the hyperbolic form on $\mathbb{H}(R)^m$, the nilpotency class 
 of the group $\frac{{\rm SO}_{(R[X],(X))}(M[X])}{{\rm EO}_{(R[X],(X))}(M[X])}$ at most 2.
\end{theorem}
\begin{proof}
	Let $\alpha(X)\in {\rm SO}_{(R[X],(X))}(M[X]).$ Then $\alpha(0)=Id.$ Define $\gamma(X):=[\alpha(X),\beta(X)]$ for $\beta(X) \in [{\rm SO}_{(R[X],(X))}(M[X]),{\rm SO}_{(R[X],(X))}(M[X]).$  Clearly $\gamma(0)=Id.$ Then for any maximal ideal $\mathfrak{m}$ of $R$, we have, $\beta(X)_{\mathfrak{m}}\in [{\rm SO}_{n+2m}(R_{\mathfrak{m}}[X],(X)_{\mathfrak{m}}),{\rm SO}_{n+2m}(R_{\mathfrak{m}}[X],(X)_{\mathfrak{m}})]$. Thus by Theorem \ref{nilpotency}, we get $\gamma(X)_{\mathfrak{m}}=[\alpha(X)_{\mathfrak{m}},\beta(X)_{\mathfrak{m}}] \in {\rm EO}_{n+2}(R_{\mathfrak{m}}[X],(X)_{\mathfrak{m}})= {\rm EO}_{(R_{\mathfrak{m}}[X],(X)_{\mathfrak{m}})}(M[X]).$  Now by Lemma \ref{RLG}, we have $\gamma(X) \in {\rm EO}_{(R[X],(X))}(M_{\mathfrak{m}}[X])$. 
\end{proof}
 Now Theorem \ref{DSERNilpotency} follows naturally. Proof works same as in \ref{DSERnilpotencycase1}.

\section{Commutativity principle for odd elementary orthogonal group}

In \cite{RRaoSS2022}, S. Sharma and R. A. Rao studied the commutativity principle for even dimensional orthogonal group. The principle states that, over a local ring, the even orthogonal commutator is in the elementary orthogonal group. They extended these results over an arbitrary ring by proving that the commutator of an even orthogonal matrix and an orthogonal matrix homotopic to identity is in the elementary orthogonal group. Here we prove analogous results for odd orthogonal group.

\begin{lemma}\label{LGPEqui}
	Let $(Q,q)$ be a non-degenerate quadratic $R$-space of rank $n\geq 1$ and $P$ be a finitely generated projective $R$-module of rank $m\geq 2$. Let $M=Q\perp \mathbb{H}(P)$. Assume that, for  $\alpha(X) \in {\rm O}_{R[X]}(M[X])$, we have $\alpha_{\mathfrak{m}}(X)\in {\rm O}_{R_{\mathfrak{m}}}(M_{\mathfrak{m}})\cdot {\rm EO}_{R_{\mathfrak{m}}[X]}(M_{\mathfrak{m}}[X])$,  for each maximal ideal $\mathfrak{m}$ of $R$. Then $\alpha(X) \in {\rm O}_{R}(M)\cdot{\rm EO}_{R[X]}(M[X])$. In particular, if $\alpha(0)\in {\rm EO}_{R}(Q,\mathbb{H}(P))$, then $\alpha(X) \in {\rm EO}_{R[X]}(M[X])$. 
\end{lemma}

\begin{proof}

	First part of the lemma is proved in \cite[Theorem~5.1]{AARR2014}. In particular, suppose  $\alpha(0) \in  {\rm EO}_{R}(Q,\mathbb{H}(P))$ for $\alpha(X)\in {\rm O}_{R}(M)\cdot{\rm EO}_{(R[X])}(M[X])$. Then $\alpha(X)=\alpha'\cdot\alpha''(X)$ for some $\alpha' \in {\rm O}_{R}(M)$ and $\alpha''(X) \in {\rm EO}_{R[X]}(M[X])$. But $\alpha(0)=\alpha'\cdot\alpha''(0) \in {\rm EO}_{R}(M)$. Thus $\alpha' \in {\rm EO}_{R}(Q,\mathbb{H}(P))$ which implies that $\alpha(X)\in {\rm EO}_{(R[X])}(M[X])$.
 
\end{proof}
\begin{theorem}
	Let $R$ be a local ring. Then for $m\geq 3$, we have 
 $$[{\rm O}_{2m-1}(R[X]),{\rm O}_{2m-1}(R[X])]  \perp I_2 \subseteq {\rm EO}_{2m+1}(R[X]).$$
 
\end{theorem}

\begin{proof}

	Let $\alpha(X), \beta(X) \in {\rm O}_{2m-1}(R[X])$. Define $\gamma(X,T):=[\alpha(TX)\perp I_2,\ \beta(X)\perp I_2]$. Then for every maximal ideal $\mathfrak{m}$ of $R[X],$ we have $\gamma(X,T)_\mathfrak{m}=[\alpha(TX)_\mathfrak{m}\perp I_2,\ \beta(X)_\mathfrak{m}\perp I_2]$. By Lemma \ref{splitting}, we get $(\beta(X)\perp I_2)_{\mathfrak{m}} = (I_{2m-1} \perp \rho(X))(\tau(X))$, for some $\rho(X) \in {\rm O}_2(R[X]_{\mathfrak{m}})$ and $\tau(X) \in {\rm EO}_{2m+1}(R[X]_{\mathfrak{m}}[T])$. Therefore, we get 
  $$\gamma(X,T)_{\mathfrak{m}} = (\alpha(TX)_\mathfrak{m}\perp I_2)(I_{2m-1} \perp \rho(X))\tau(X)(\alpha(TX)_\mathfrak{m}^{-1}\perp I_2)\tau(X)^{-1}(I_{2m-1}\perp \rho(X)^{-1}).$$
 Clearly, $\gamma(X,T)_{\mathfrak{m}} \in {\rm EO}_{R}(M_{\mathfrak{m}}[X])$. Now, consider $\alpha(0)\perp I_2 =(I_{2m-1}\perp \rho)\tau$, where $\rho \in {\rm O}_2{(R[X])}$ and $\tau \in {\rm EO_{2m+1}(R)}$. Then, by normality of the elementary orthogonal subgroup ${\rm EO}_{2m+1}(R[X])$ in ${\rm O}_{2m+1}(R[X])$, we have $$\gamma(X,0) = (\alpha(0)\perp I_2)(I_{2m-1} \perp \rho)\tau (\alpha(0)^{-1}\perp I_2)\tau^{-1}(I_{2m-1}\perp \rho^{-1})\in {\rm EO}_{2m+1}(R[X]).$$ On account of Lemma \ref{LGPEqui}, we get $\gamma(X,T) \in {\rm EO}_{2m+1}(R[X,T])$. In particular, for $T=1,$ we have, $[\alpha(X)\perp I_2,\beta(X)\perp I_2] \in {\rm EO}_{2m+1}(R[X])$.
 
\end{proof}

Now we prove the odd orthogonal group over an arbitrary commutative ring with $2$ is invertible.

{\bf Proof of theorem \ref{homotopytheorem}.}

\begin{proof}

	Let $\alpha \in {\rm SO}_{2m-1}(R)$ be homotopic to $Id$ and $\beta \in {\rm O}_{2m-1}(R)$. Then there exists $\alpha(X)\in  {\rm SO}_{2m-1}(R[X])$ such that $\alpha(0)=Id$ and $\alpha(1)=\alpha $. Define $\gamma(X):=[\alpha(X) \perp I_2,\beta \perp I_2]$. Then for all maximal ideal $\mathfrak{m}$ of $R$, $\gamma(X)_{\mathfrak{m}}=[\alpha(X)\perp I_2)_{\mathfrak{m}},(\beta\perp I_2)_{\mathfrak{m}}]$. By Lemma \ref{splitting}, we get $(\beta \perp I_2)_{\mathfrak{m}}=(I_{2m-1} \perp \rho)\tau$, for some $\rho \in {\rm O}_2(R_{\mathfrak{m}})$ and $\tau \in {\rm EO}_{2m-1}(R_{\mathfrak{m}})$. Thus $\gamma(X)_{\mathfrak{m}} \in {\rm EO}_{2m+1}(R_{\mathfrak{m}}[X])$. Also, we have  $\gamma(0)=Id$. Thus by Lemma \ref{ALG}, we get $\gamma(X) \in {\rm EO}_{2m+1}(R[X])$ and putting $X=1, \gamma(1) = [\alpha\perp I_2,\beta \perp I_2] \in {\rm EO}_{2m+1}(R)$.
 
\end{proof}

	\noindent \textbf{Acknowledgment:}
	The first author would like to thank KSCSTE Young Scientist Award Scheme (2021-KSYSA-RG), Govt. of Kerala and RUSA, Govt. of India,  for providing grant to support this work. She would also like to thank SERB for SURE grant (SUR/2022/004894). The second author would like to thank the Council of Scientific and Industrial Research (CSIR), India (09/0239(13173)/2022-EMR-I) for the research fellowship. The authors would like to thank Ms. Aparna Pradeep V. K. and Dr. S. Sharma for their valuable suggestions and discussions.

	\bibliographystyle{elsearticle-num}

\end{document}